\documentclass[11pt,a4paper]{amsart}

\usepackage{epstopdf}
\usepackage[caption=false]{subfig}
 \usepackage{ifpdf}
\usepackage[numbers,sort&compress]{natbib}
\usepackage{lineno,hyperref}
\usepackage{amssymb,amsmath,mathtools}
\usepackage{lscape, setspace}
\usepackage{enumerate}
\usepackage{multirow}
\usepackage{booktabs}
\usepackage{latexsym}
\usepackage{amscd,amsfonts}
\usepackage{graphicx}
\usepackage{color}
\usepackage{pdflscape}
\usepackage{longtable}
\usepackage{epsfig}
\usepackage{tikz}

\usepackage{adjustbox}
\usetikzlibrary{arrows}

\DeclareMathOperator{\ad}{ad}

\DeclareMathOperator{\Aut}{Aut}

\DeclareMathOperator{\End}{End}
\DeclareMathOperator{\GL}{GL}

\DeclareMathOperator{\rk}{rk}
\DeclareMathOperator{\Span}{Span}

\DeclareMathOperator{\tr}{tr}

\DeclareMathOperator{\Mat}{Mat}

\def \[{[\![}
\def \]{]\!]}

\newcommand{\g}{\mathfrak g}
\newcommand{\h}{\mathfrak h}

\newcommand{\R}{\mathbb{R}}
\newcommand{\C}{\mathbb{C}}
\newcommand{\F}{\mathbb{F}}
\newcommand{\Z}{\mathbb{Z}}

\newcommand \ab {\operatorname{ab}}

\newcommand \LS {\mathcal{LS}^{(m,n)}}

\theoremstyle{plain}
\newtheorem{theorem}{Theorem}[section]
\newtheorem{lemma}[theorem]{Lemma}

\theoremstyle{definition}
\newtheorem{definition}[theorem]{Definition}

\theoremstyle{remark}

\numberwithin{equation}{section}
\begin{document}


\title{On degenerations of Lie superalgebras}

\author{Mar\'ia Alejandra Alvarez}

\address{Departamento de Matem\'aticas - Universidad de Antofagasta - Chile}
\email{maria.alvarez@uantof.cl}

\author{Isabel Hern\'andez}
\address{CONACYT-CIMAT - Unidad M\'erida}
\email{ isabel@cimat.mx}


\maketitle
\begin{abstract}
We give necessary conditions for the existence of degenerations  between  two complex Lie superalgebras of dimension $(m,n)$. As an application, we study the variety  $\mathcal{LS}^{(2,2)}$ of complex Lie superalgebras  of dimension $(2,2)$. First we give the algebraic classification and then obtain that $\mathcal{LS}^{(2,2)}$ is the union of seven irreducible components, three of which are the Zariski closures of rigid Lie superalgebras. As byproduct, we obtain an example of a nilpotent rigid Lie superalgebra, in contrast to the classical case where no example is known.
\end{abstract}

\section{Introduction}

 The problem of determining the orbits closures,  or the irreducible components of varieties of algebras, have been studied for many structures: Lie algebras (\cite{B1}, \cite{B2}, \cite{BS}, \cite{Ch}, \cite{GO'H}, \cite{KN}, \cite{L}, \cite{NP}, \cite{S}, \cite{W}), Jordan algebras (\cite{ACGS}, \cite{AFM}, \cite{GKP}, \cite{KM}, \cite{KP}), Leibniz algebras (\cite{CKLO}, \cite{IKV}, \cite{KPPV}), pre-Lie algebras in \cite{BB1}, Novikov algebras in \cite{BB2},  Filippov algebras in \cite{dAIP}, binary Lie and nilpotent Malcev algebras in \cite{KPV}, superalgebras in \cite{AZ}, etc. However, we have found no literature about Lie superalgebras.

The aim of this paper is to give necessary conditions for the existence of degenera\-tions   between  two complex Lie superalgebras of dimension $(m,n)$.  For this goal,  some invariants are studied  (see  section $2.1$).  As an application,  we study the variety of $(2,2)$-dimensional Lie superalgebras, 
where the group $\GL_2 (\C) \oplus \GL_2(\C)$ acts by ``change of basis"   producing  fourteen orbits,  five orbits depending of one parameter, and one orbit depending of two parameters (see Theorem \ref{thm:alg_clas}).  After that, we obtain the Zariski closure of every orbit,  and the irreducible components of  this  variety (see Theorem \ref{2-2-irred-comp}).  Moreover, we obtain a  nilpotent  rigid  Lie superalgebra, i.e., a nilpotent Lie superalgebra whose Zariski orbit is open (see   Theorem \ref{thm:rigid}).  It is important to notice that in the classical case, there are no known examples of nilpotent rigid Lie algebras.    

%
\subsection{Preliminaries}
A  supervector space $V=V_0\oplus V_1$ over the field $\F$, is a $\Z_2$-graded  $\F$-vector space, i.e.,  a  vector space decomposed into a direct sum of subspaces. The elements on
 $V_0\setminus\{ 0 \}$ (respectively,  on $ V_1 \setminus \{0\}$) are called even (respectively, odd). Even and odd elements together are called homogeneous; the degree of an homogeneous element is
 $i$, denoted by $|v|=i$,  if $v \in V_i \setminus \{ 0 \}$, for $i \in \Z_2$. If $\dim_\F (V_0)=m$ and $\dim_\F (V_1)=n$, we say that the dimension of $V$ is $(m,n)$.   The vector space $\End(V)$ can be viewed as a  supervector space,  denoted by  $\End(V_0|V_1)$,  where  $\End( V_0 | V_1)_i = \{  T  \in \End(V) \; | \; T(V_j)\subset V_{i+j},   \;  j \in \Z_2  \}$, for $i\in \Z_2$. Given an homogeneous basis  $\{e_1,  \dots, e_{m}, f_1, \dots, f_n\}$ for  $V=V_0 \oplus V_1$, (that is, $V_0 = \Span_\F\{ e_1, \dots,
 e_m\} $ and $V_1 = \Span_\F\{ f_1, \dots,
 f_n\}$), it follows that $\End(V_0 | V_1)_i$ can be identified with $(\Mat_{(m|n)}(\F))_i$, for $ i \in \Z_2$,  where
$$
\begin{array}{l}
(\Mat_{(m|n)}(\F))_0= \left\{
\begin{psmallmatrix}
A & 0 \\
0 & D
\end{psmallmatrix} |\, A \in \Mat_m(\F),  \; D \in \Mat_n(\F)   \right\}, \quad \text{ and }
 \\
 \\
(\Mat_{(m|n)}(\F))_1= \left\{
\begin{psmallmatrix}
0 & B \\
C & 0
\end{psmallmatrix} |\, C \in \Mat_{n \times m}(\F), \; B \in \Mat_{m\times n}(\F) \right\}.
 \end{array}
$$
 In particular,  $\Aut(V_0|V_1)$ is a $\Z_2$-graded group such that $\Aut(V_0|V_1)_0$  can be identified  with   $\GL_m(\F) \oplus \GL_n(\F)$.

A Lie superalgebra  is a supervector space $\g=\g_0 \oplus \g_1$ endowed with a
bilinear map  $\[ \cdot,  \cdot \] : V \times  V \to V$   satisfying the following
\begin{enumerate}[(i)]
\item $\[\g_i, \g_j\] \subset \g_{i+j}$, for $i, j \in \Z_2$.
\item The super skew-symmetry:   $\[x,y\]= -(-1)^{|x||y|}\[y, x\]$.
\item The super Jacobi identity:
 $$(-1)^{|x||z|}\[ \[ x,y    \], z    \] +(-1)^{|x||y|}\[ \[ y,z    \], x    \] +   (-1)^{|y||z|}\[ \[ z,x    \], y    \] =0$$
\end{enumerate}
for  $x,y,z \in (\g_0 \cup \g_1) \setminus \{ 0 \}$.  A linear map  between Lie superalgebras $\Phi: \g \to \g^\prime$  is called a Lie superalgebra morphism if $\Phi$ is even (i.e.  $\Phi(\g _i) \subset    \g_i $,  for $ i \in \Z_2$), and
$\Phi (\[ x,y\] )= \[ \Phi(x),  \Phi(y) \] ^\prime$, for $x,y \in \g$.
 (see \cite{Sc} for standard terminology on Lie superalgebras).

Let  $V=V_0\oplus V_1$ be a  complex $(m,n)$-dimensional  supervector space with a fixed homogeneous bases
$\left\{e_1,\dots,e_m, f_1,\dots,f_n\right\}$.  Given a Lie superalgebra structure $\[ \cdot, \cdot \]$ on $V$,  we can identify  $ \g=(V , \,  \[ \cdot, \cdot \])$  with its set of structure cons\-tants 
$\left\{c_{ij}^k,\rho_{ij}^k,\Gamma_{ij}^k\right\}\in\C^{m^3+2mn^2}$, where
$$\[e_i,e_j\]=\sum_{k=1}^mc_{ij}^ke_k,\quad \[e_i,f_j\]=\sum_{k=1}^n\rho_{ij}^kf_k,\quad\text{and}\quad \[f_i,f_j\]=\sum_{k=1}^m\Gamma_{ij}^ke_k.$$
Since every set of structure constants must satisfy the polynomial equations given by the super skew-symmetry  and the super Jacobi identity,  the set of all Lie superalgebras of dimension $(m,n)$ is an affine variety in $\C^{m^3+2nm^2}$, denoted by $\LS$. Notice that every point $\{ c_{ij}^k, \rho_{ij}^k , \Gamma_{ij}^k \} \in \LS$ represents a  Lie superalgebra of dimension $(m,n)$.

On the other hand,  since Lie superalgebra isomorphisms are  even maps, it follows that  the group $G=\GL_m(\C)\oplus\GL_n(\C)$ acts on $\LS$ by ``change of basis'': $g\cdot\[x,y\]=g\[g^{-1}x,g^{-1}y\]$,  for $g\in G$ and $x,y\in \g$. Observe that the set  of $G$- orbits of this action  are in one to one correspondence with the isomorphism classes in $\LS$.

\section{Degenerations on the variety $\LS$}

We recall some known facts for irreducible varieties. A nonempty algebraic set is said to be irreducible if it cannot be written as the union of two proper, nonempty, algebraic subsets.  Also, every nonempty algebraic set  $Y$ can be written  as  the union of a  finite number of  maximal irreducible algebraic subsets  $X = X_1\cup\dots\cup X_r$; these are called the {\bf irreducible components} of $X$. Notice that such decomposition is unique.

As a first step  for finding such  irreducible components we use the following result:

\begin{lemma}
If $G$ is a connected algebraic group acting on a variety $X$, then the irreducible components of $X$ are stable under the action of $G$. Moreover, the irreducible components of the variety $X$ are closures of single orbits or closures of infinite families of orbits.
\end{lemma}

\begin{definition}\label{def:rigid}
An element $x\in X$ is called \textit{rigid}, if its orbit $\mathcal{O}(x)$ is open in $X$.
\end{definition}

Rigid elements of the variety are important due to the fact that if $x$ is rigid in $X$, then there exists an irreducible component $\mathcal{C}$ such that $\mathcal{C}\cap\mathcal{O}(x)$ is a non-empty open in $\mathcal{C}$ and thus $\mathcal{C}\subset\overline{\mathcal{O}(\g)}$.

\bigskip

Consider now the variety $\LS$. Given two Lie superalgebras $\g$ and $\h$, we say that  $\g$ \emph{degenerates} to $\h$, denoted by $\g\rightarrow \h$,  if $\h$ lies in the Zariski closure of the $G$-orbit $\mathcal{O}(\g)$, where $G=\GL_m(\C)\oplus\GL_n(\C)$. 
Since each orbit $\mathcal{O}(\g)$ is a constructible set, its closures relative to the Euclidean and the Zariski topologies are the same  (see \cite{M}, 1.10  Corollary 1, p. 84).  As a  consequence  the following  is obtained.

\begin{lemma}
Let $\C(t)$ be the field of fractions of the polynomial ring $\C[t]$. If there exists an operator $g_t\in\GL_m(\C(t))\oplus\GL_n(\C(t))$ such that $\displaystyle\lim_{t\rightarrow 0}g_t\cdot\g=\h$, then $\g\rightarrow\h$.
\label{def-equiv-defor}
\end{lemma}

Given $\g=\g_0\oplus\g_1 \in \LS$, recall the identification of  $\g$ with its set of structure cons\-tants, $\g=\left\{c_{ij}^k,\rho_{ij}^k,\Gamma_{i,j}^k\right\}$. From the previous Lemma,  it follows that  every Lie superalgebra $\g\in\LS$ degenerates to the Lie superalgebra $\mathfrak{a}=\{0,0,0\}$. In fact, take  $g_t=t^{-1}\left(I_{m}\oplus I_{n}\right)$, where $I_{k}$ is the identity map in $\GL_{k}(\C)$,  then  $\displaystyle\lim_{t\rightarrow 0}g_t\cdot\g=\frak a$. Thus, $ \g \rightarrow  \mathfrak{a}$. 

\subsection{Invariants}

Next we list the invariants for the variety $\LS$ that will be used to obtain non-degenerations in the variety. Most of these invariants have already been used in the study of the variety of Lie algebras. Given  $\g \in \LS$, we use the notation  $\Gamma_\g$ for $\left(\Gamma_{ij}^k\right)$. On the other hand,  $\mathfrak{D}(\alpha,\beta,\gamma)(\g)$ denotes the space of $(\alpha,\beta,\gamma)$-derivations for the Lie superalgebra $\g$ (see \cite{ZZ} for more details).

We define the  map  $\mathcal{A}:\LS\rightarrow \LS$ by $\left\{c_{ij}^k,\rho_{ij}^k,\Gamma_{i,j}^k\right\}\mapsto\left\{c_{ij}^k,0,0\right\}$. Notice that the map $\mathcal{A}$ is a morphism that `forgets' the Lie superalgebra structure on $\g$  and retrieves the Lie (super)algebra $\h=\g_0\oplus\C^n  
\hookrightarrow \LS$. It is easy to see  that, the map $\mathcal{A}$ is $G$-equivariant, that is, 
 $\mathcal{A}\left(g\cdot\left\{c_{ij}^k,\rho_{ij}^k,\Gamma_{ij}^k\right\}\right)\\=g\cdot \mathcal{A}\left(\left\{c_{ij}^k,\rho_{ij}^k,\Gamma_{ij}^k\right\}\right)$, for  $g\in G$.  Moreover, $$\mathcal{A}\left(\overline{G\cdot\left\{c_{ij}^k,\rho_{ij}^k,\Gamma_{ij}^k\right\}}\right)=\overline{G\cdot\left\{c_{ij}^k,0,0\right\}}. $$
As  a consequence we obtain the following  useful result:

\begin{lemma}\label{lem:deg_alg}
Let $ \g=\g_0\oplus\g_1,\ \h=\h_0\oplus\h_1 \in \LS$. If $\g \rightarrow\h$,  then  $\g_0\rightarrow\h_0$.
\end{lemma}

Also, we have the following relations:

\begin{lemma}\label{lem:invariants}
Let $\g,\h\in\LS$. If $\g\rightarrow\h$,  then the following relations must hold:
\begin{enumerate}[(i)]
\item $\dim \mathcal{O}(\g)>\dim\mathcal{O}(\h)$.
\item If $\Gamma_\g\equiv 0$ then $\Gamma_\h\equiv 0$.
\item $\dim(\g^1)_i\geq\dim(\h^1)_i$ for $i\in\Z_2$, where $\g^1=[\g,\g]$.
\item $\dim \mathfrak{D}(\alpha,\beta,\gamma)_i(\g)\leq\dim \mathfrak{D}(\alpha,\beta,\gamma)_i(\h)$,  for $i\in\Z_2$.
\item Let $\mathfrak{B}$ and $\mathfrak{B}'$ be fixed bases for $\g$ and $\h$, respectively. If $\tr\ad(x)=0$ for every $x\in\mathfrak{B}$,  then $\tr\ad(y)=0$,  for every $y\in\mathfrak{B}'$.
\end{enumerate}
\end{lemma}

\begin{proof}
Item $(i)$  follows by the Closed Orbit Lemma (see \cite{B} I. Lemma 1.8, p. 53). The remainning follow by proving that the corresponding sets are closed, using in some cases the upper semi continuity of appropriated functions (see \cite{C}, \S 3 Theorem 2, p. 14).
\end{proof}

\begin{definition}
Let $\g=\left\{c_{ij}^k, \rho_{ij}^k,\Gamma_{ij}^k\right\}$ be a Lie superalgebra in $\LS$.
\begin{enumerate}
\item The 
Lie superalgebra $\ab(\g)$ is  defined by $\ab(\g)= \left\{0, 0 ,\Gamma_{ij}^k\right\} $.
\item The 
Lie (super)algebra $\mathcal{F}(\g)$ is defined by $\mathcal{F}(\g) = \left\{c_{ij}^k,\rho_{ij}^k, 0 \right\}$. 
\end{enumerate}
\end{definition}

\begin{lemma}\label{lem:abel}
Let $\g,\h\in\LS$. If $\g\rightarrow\h$,  then 
\begin{enumerate}[(i)]
\item $\ab(\g)\rightarrow\ab(\h)$, 
\item $\mathcal{F}(\g)\rightarrow \mathcal{F}(\h)$. 
\end{enumerate}
\end{lemma}

\begin{proof}
Let $\{e_1,\dots,e_m,f_1,\dots,f_n\}$ be an homogeneous basis for the underlying supervector space $V=V_0\oplus V_1$ of $\g$ and $\h$. Since  $\g\rightarrow\h$  it follows that there exists a map $g_t\in\GL_m(\C(t))\oplus\GL_n(\C(t))$ such that $\displaystyle\lim_{t\rightarrow 0}g_t\cdot{\g}={\h}$. In particular
\begin{enumerate}[(i)]
\item $\displaystyle\lim_{t\rightarrow 0}g_t\[g_t^{-1}f_i,g_t^{-1}f_j\]_{\g}=\[f_i,f_j\]_{\h}$. Hence,  $\ab(\g)\rightarrow\ab(\h)$.
\item $\displaystyle\lim_{t\rightarrow 0}g_t\[g_t^{-1}e_i,g_t^{-1}e_j\]_{\g}=\[e_i,e_j\]_{\h}\,$ and $\,  \displaystyle\lim_{t\rightarrow 0}g_t\[g_t^{-1}e_i,g_t^{-1}f_j\]_{\g}=\[e_i,f_j\]_{\h}$. Therefore,  $\mathcal{F}(\g)\rightarrow \mathcal{F}(\h)$.
\end{enumerate}
\end{proof}
%
%

{\bf The $(i,j)$-invariant.} Let $\g \in \LS$ and $(i,j)$ be a pair of positive integers. Let 
$$c_{i,j}=\frac{\tr(\ad x)^i\tr(\ad y)^j}{\tr\left((\ad x)^i\circ(\ad y)^j\right)}$$
be a quotient of two polynomials in the structure constants of $\g$, for all $x,y\in\g$ such that 
both polynomials are not zero. If  $c_{i,j}$ is independent of the choice of $x$ and $y$,  then it is an invariant called the {\bf $(i,j)$-invariant of $\g$}.

\begin{lemma}\label{lemma:ij-inv}
Let $\g,\h\in\LS$. If $\g\rightarrow\h$ and the $(i,j)$-invariant exists for both Lie superalgebras, then $\g$ and $\h$ have the same $(i,j)$-invariant.
\end{lemma}

\section{The variety ${\mathcal{LS}}^{(2,2)}$}
Lie superalgebras of dimension 4 over $\R$ were classified in \cite{Back}. However, Backhouse's classification does not include the classification of $(2,2)$-dimensional Lie superalgebras over $\C$. We provide in this section a complete classification, up to isomorphism, of all complex Lie superalgebras of dimension $(2,2)$.  The representatives of the isomorphism classes are denoted by
${\mathcal  LS}_n$,  where $n\in \{ 0, \dots, 19\}$. We write the pro\-ducts for each of them,  in terms of an homogeneous basis $\{ e_1, e_2, f_1, f_2 \} $,  where
$({\mathcal  LS_n})_0 = \Span \{ e_1, e_2 \}$ and ${(\mathcal  LS_n})_1 = \Span \{ f_1, f_2 \}$. Some of these Lie superalgebras depend on one or two  complex parameters, producing non-isomorphic families of Lie superalgebras as the following Theorem shows. 

\begin{theorem}\label{thm:alg_clas}
Let $\g$ be a  complex Lie superalgebra of dimension $(2,2)$. Then, $\g$ is isomorphic to one and only one of the following Lie superalgebras:

$$\begin{array}{llll}
\mathcal {LS}_{0}: &\[\cdot,\cdot \]=0. &&\\
\mathcal {LS}_{1}:& \[ f_1, f_1 \]= e_1, & \[f_2, f_2 \]=e_2.& \\
\mathcal {LS}_{2}:& \[ f_1, f_1 \]= e_1, & \[f_2, f_2 \]=e_1. &\\
\mathcal {LS}_{3}:& \[ f_1, f_1 \]= e_1.&&\\
\mathcal {LS}_{4}: & \[ f_1, f_2 \] = e_1, & \[f_2, f_2\]=e_2.\\
\mathcal {LS}_{5}:& \[ e_1, f_1 \]= f_1,  & \[e_2, f_2 \]=f_2.&   \\
\mathcal {LS}_{6}^{\alpha}:& \[ e_2, f_1 \]= f_1, &  \[e_2, f_2 \]=\alpha f_2.    \\
\mathcal {LS}_{7}:& \[ e_2, f_1 \]= f_1, & \[e_2, f_2 \]=-f_2,  &\[f_1, f_2   \]=e_1.   \\
\mathcal {LS}_{8}:& \[ e_2, f_1 \]= f_1, & \[f_2, f_2 \]=e_1.   &   \\
\mathcal {LS}_{9}:& \[ e_1, f_1 \]= f_1, & \[e_1, f_2 \]=f_2,&  \[e_2, f_2\]=f_1.     \\
\mathcal {LS}_{10}:  & \[e_2, f_1 \]=f_1, & \[ e_2,f_2\]= f_1+f_2.    \\
\mathcal {LS}_{11}: & \[e_2, f_2 \] = f_1. & & \\
\mathcal {LS}_{12}: & \[ e_2, f_2\] = f_1, & \[ f_2, f_2 \] = e_1. \\
\mathcal {LS}_{13}^{\alpha,\beta}: & \[ e_1, e_2 \] =e_1,&    \[ e_2, f_1\] = \alpha f_1, &  \[ e_2, f_2\] = \beta f_2. \\
\mathcal {LS}_{14}^{\alpha}: & \[ e_1, e_2 \] =e_1,&    \[ e_2, f_1\] = \alpha f_1, &  \[ e_2, f_2\] =-(\alpha+1) f_2,    \\&\[ f_1, f_2  \]= e_1,  &   & \\
\mathcal {LS}_{15}^{\alpha}: & \[ e_1, e_2 \] =e_1,&    \[ e_2, f_1\] = \alpha f_1, &  \[ e_2, f_2\] = -\frac{1}{2} f_2, \\
          & \[ f_2, f_2  \]= e_1. & & \\
\mathcal{LS}_{16}^{\alpha}: & \[ e_1, e_2 \] =e_1,&    \[ e_2, f_1\] = \alpha f_1,&  \[ e_2, f_2\] = f_1+ \alpha f_2. \\
\mathcal{LS}_{17}: & \[ e_1, e_2 \] =e_1,&    \[ e_2, f_1\] =  - \frac{1}{2}  f_1, &  \[ e_2, f_2\] = f_1 - \frac{1}{2} f_2. \\
  & \[ f_2, f_2  \]= e_1. &  & \\
\mathcal{LS}_{18}^{\alpha}:  & \[ e_1, e_2 \] =e_1,& \[ e_1, f_2 \]=f_1, &  \[ e_2, f_1\] = \alpha f_1, \\
  &  \[ e_2, f_2\] = (\alpha+1) f_2. &   &    \\
 \mathcal{LS}_{19}:  & \[ e_1, e_2 \] =e_1,& \[ e_1, f_2 \]=f_1, &  \[ e_2, f_1\] = -f_1, \\
  &  \[ f_1, f_2   \]= e_1,   & \[ f_2, f_2  \]= 2e_2, &  \\
\end{array}$$
where $\alpha, \beta \in \C$. Furthermore, $\mathcal {LS}_{n}^{\alpha} \simeq  \mathcal {LS}_{n}^{\alpha'}$ if and only if $\alpha= \alpha'$, for $n\in \{6,15,16, 18\}$; $\mathcal {LS}_{14}^{\alpha} \simeq \mathcal {LS}_{14}^{\alpha'}$ if and only if either $\alpha=\alpha'$ or $\alpha+\alpha'=-1$; and $\mathcal {LS}_{13}^{\alpha,\beta} \simeq \mathcal {LS}_{13}^{\alpha', \beta'}$ if and only if $\{\alpha, \beta\}= \{\alpha', \beta'\}$.
\end{theorem}

\subsection{Invariants in the variety ${\mathcal{LS}}^{(2,2)}$}
In order to obtain all possible degenerations in ${\mathcal{LS}}^{(2,2)}$,  we list some of the invariants for every $(2,2)$-dimensional Lie superalgebra. We list all Lie superalgebras from higher to lower orbit dimension. 



{\scriptsize
\begin{longtable}{|c|c|c|c|c|}
\caption[]{Invariants}
\label{Table1}\\
\hline
$\g$ & $\dim\mathcal{O}(\g)$ & $\rk\Gamma_\g$ &  $(\dim(\g^1)_0,\dim(\g^1)_1)$ & $c_{i,j}(\g)$ \\ 
\hline
\endfirsthead
\caption[]{(continued)}\\
\hline
$\g$ & $\dim\mathcal{O}(\g)$ & $\rk\Gamma_\g$ &  $(\dim(\g^1)_0,\dim(\g^1)_1)$ & $c_{i,j}(\g)$ \\ \hline
\endhead

{$\mathcal{LS}_1$} & {6} & {2} &  {(2,0)} & {$\not\exists$} \\ [1.5mm] \hline
{$\mathcal{LS}_5$} & {6} & {0} &  {(0,2)} & {$\not\exists$} \\ [1.5mm] \hline
$\mathcal{LS}_{19}$ & 6 & 2 &  (2,1) & $c_{2i,2j}=2$   \\ [1.5mm] \hline
{$\mathcal{LS}_4$} & {5} & {2} &  {(2,0)} & {$\not\exists$}\\ [1.5mm] \hline
{$\mathcal{LS}_7$} & {5} & {1} &  {(1,2)} & {$c_{2i,2j}=2$} \\ [1.5mm] \hline
{$\mathcal{LS}_8$} & {5} & {1} &  {(1,1)} & {$c_{i,j}=1$} \\ [1.5mm] \hline
{$\mathcal{LS}_9$} & {5} & {0} &  {(0,2)} & {$c_{i,j}=2$} \\ [1.5mm] \hline
$\mathcal{LS}_{14}^{\alpha}$ & 5 & 1 & $(1,2-\delta_{0,\alpha}-\delta_{1,\alpha})$ & $c_{i,j}=\frac{\left((-1)^i+\alpha^i+\left(-\alpha-1\right)^i\right)\left((-1)^j+\alpha^j+\left(-\alpha-1\right)^j\right)}{(-1)^{i+j}+\alpha^{i+j}+\left(-\alpha-1\right)^{i+j}}$  \\ [1.5mm] \hline
$\mathcal{LS}_{15}^{\alpha\neq-\frac{1}{2}}$ & 5 & 1 &  $(1,2-\delta_{0,\alpha})$ & $c_{i,j}=\frac{\left((-1)^i+\alpha^i+\left(-\frac{1}{2}\right)^i\right)\left((-1)^j+\alpha^j+\left(-\frac{1}{2}\right)^j\right)}{(-1)^{i+j}+\alpha^{i+j}+\left(-\frac{1}{2}\right)^{i+j}}$ \\ [1.5mm] \hline
$\mathcal{LS}_{17}$ & 5 & 1 &  (1,2) & $c_{i,j}=\frac{\left((-1)^i+2\left(-\frac{1}{2}\right)^i\right)\left((-1)^j+2\left(-\frac{1}{2}\right)^j\right)}{(-1)^{i+j}+2\left(-\frac{1}{2}\right)^{i+j}}$   \\ [1.5mm] \hline
$\mathcal{LS}_{18}^{\alpha}$ & 5 & 0 &  (1,$2-\delta_{\alpha,-1}$) & $c_{i,j}=\frac{\left((-1)^i+\alpha^i+(1+\alpha)^i\right)\left((-1)^j+\alpha^j+(1+\alpha)^j\right)}{(-1)^{i+j}+\alpha^{i+j}+(1+\alpha)^{i+j}}$  \\ [1.5mm] \hline
{$\mathcal{LS}_2$} & {4} & {1} &  {(1,0)} & {$\not\exists$} \\ [1.5mm] \hline
{$\mathcal{LS}_6^{\alpha\neq 1}$} & {4} & {0} & {$(0,2-\delta_{0,\alpha})$} & {$c_{(1+\delta_{-1,\alpha})i,(1+\delta_{-1,\alpha})j}=\frac{\left(1+\alpha^i\right)\left(1+\alpha^j\right)}{1+\alpha^{i+j}}$} \\ [1.5mm] \hline
{$\mathcal{LS}_{10}$} & {4} & {0} &  {(0,2)} & {$c_{i,j}=2$} \\ [1.5mm] \hline
{$\mathcal{LS}_{12}$} & {4} & {1} &  {(1,1)} & {$\not\exists$} \\ [1.5mm] \hline
$\mathcal{LS}_{13}^{\alpha\neq\beta}$ & 4 & 0 & $(1,2-\delta_{0,\alpha}-\delta_{0,\beta})$ & $c_{i,j}=\frac{((-1)^i+\alpha^i+\beta^i)((-1)^j+\alpha^j+\beta^j)}{(-1)^{i+j}+\alpha^{i+j}+\beta^{i+j}}$  
\\ [1.5mm] \hline
$\mathcal {LS}_{15}^{-\frac{1}{2}}$ & 4 & 1 &  (1,2) & $c_{i,j}=\frac{\left((-1)^i+2\left(-\frac{1}{2}\right)^i\right)\left((-1)^j+2\left(-\frac{1}{2}\right)^j\right)}{(-1)^{i+j}+2\left(-\frac{1}{2}\right)^{i+j}}$ \\ [1.5mm] \hline
$\mathcal{LS}_{16}^{\alpha}$ & 4 & 0 &  $(1,2-\delta_{0,\alpha})$& $c_{i,j}=\frac{((-1)^i+2\alpha^i)((-1)^j+2\alpha^j)}{(-1)^{i+j}+2\alpha^{i+j}}$  \\ [1.5mm] \hline
{$\mathcal{LS}_3$} & {3} & {1} &  {(1,0)} & {$\not\exists$} \\ [1.5mm] \hline
{$\mathcal{LS}_{11}$} & {3} & {0} &  {(0,1)} & {$\not\exists$} \\ [1.5mm] \hline
{$\mathcal{LS}_6^1$} & {2} & {0} &  {(0,2)} & {$c_{i,j}=2$} \\ [1.5mm] \hline
$\mathcal{LS}_{13}^{\alpha,\alpha}$ & 2 & 0 &  $(1,2-2\delta_{0,\alpha})$ & $c_{i,j}=\frac{((-1)^i+2\alpha^i)((-1)^j+2\alpha^j)}{(-1)^{i+j}+2\alpha^{i+j}}$ \\ [1.5mm] \hline
{$\mathcal{LS}_0$} & {0} & {0} & {(0,0)} & {$\not\exists$}\\ [1.5mm] \hline

\end{longtable}}



\normalsize

\subsubsection{Non-degenerations}

Most of the non-degeneration arguments can be seen from Table \ref{Table1} and by using Lemmas \ref{lem:deg_alg}, \ref{lem:invariants} and \ref{lemma:ij-inv}
. We list here the non-degenerations obtained by comparing $(\alpha,\beta,\gamma)$-derivations or by using Lemma \ref{lem:abel}. In every case, $\g$ represents the Lie superalgebra on the left of the arrow, and $\h$ represents the Lie superalgebra on the right of the arrow.

{\scriptsize
\begin{longtable}{|l|l|}
\caption[]{Non-degenerations}
\label{Table2}\\
\hline
Non-degeneration $\g\to\h$ & Reason \\ 
\hline
\endfirsthead

$\mathcal{LS}_{13}^{1,\frac{1}{\gamma}}\not\rightarrow \mathcal{LS}_{13}^{\gamma,\gamma}$, $\gamma\neq -1$ & $\dim\mathfrak{D}\left(1,1,-1\right)_1(\g)=1>0=\dim\mathfrak{D}\left(1,1,-1\right)_1(\h)$ \\

$\mathcal{LS}_{13}^{\alpha,\beta}\not\rightarrow \mathcal{LS}_{6}^1$ & $\dim\mathfrak{D}\left(\frac{1}{\alpha},1,-1\right)_1(\g)=1>0=\dim\mathfrak{D}\left(\frac{1}{\alpha},1,-1\right)_1(\h)$\\ 

$\mathcal{LS}_{13}^{1,2}\not\rightarrow \mathcal{LS}_{13}^{\frac{1}{2},\frac{1}{2}}$ & $\dim\mathfrak{D}\left(1,1,-1\right)_1(\g)=1>0=\dim\mathfrak{D}\left(1,1,-1\right)_1(\h)$\\

$\mathcal{LS}_{13}^{1,-2}\not\rightarrow \mathcal{LS}_{13}^{-\frac{1}{2},-\frac{1}{2}}$ & $\dim\mathfrak{D}\left(1,1,-1\right)_1(\g)=1>0=\dim\mathfrak{D}\left(1,1,-1\right)_1(\h)$\\

$\mathcal{LS}_{13}^{1,-\frac{1}{2}}\not\rightarrow \mathcal{LS}_{13}^{-2,-2}$ & $\dim\mathfrak{D}\left(1,1,-1\right)_1(\g)=1>0=\dim\mathfrak{D}\left(1,1,-1\right)_1(\h)$\\

$\mathcal{LS}_{13}^{\alpha,-\frac{1}{2}}\not\rightarrow \mathcal{LS}_{6}^1$ & $\dim\mathfrak{D}\left(-2,1,-1\right)_1(\g)=1>0=\dim\mathfrak{D}\left(-2,1,-1\right)_1(\h)$\\ \hline


$\mathcal{LS}_9\not\rightarrow \mathcal{LS}_6^{-1}$ & $\dim\mathfrak{D}(1,1,0)_1(\g)=1>0=\dim\mathfrak{D}(1,1,0)_1(\h)$\\ \hline


$\mathcal{LS}_{14}^{\alpha}\not\rightarrow \mathcal{LS}_{13}^{\beta,\beta+1},\ \mathcal{LS}_{13}^{\beta,-\frac{1}{2}},\ \mathcal{LS}_6^1$ & $\mathcal{F}(\g)\simeq\mathcal{LS}_{13}^{\alpha,-(\alpha+1)}\not\rightarrow \h=\mathcal{F}(\h)$\\

$\mathcal{LS}_{14}^1\not\rightarrow \mathcal{LS}_{16}^{-\frac{1}{2}},\mathcal {LS}_{15}^{-\frac{1}{2}}$ & $\dim\mathfrak{D}\left(-\frac{1}{2},1,-1\right)_1(\g)=1>0=\dim\mathfrak{D}\left(-\frac{1}{2},1,-1\right)_1(\h)$\\

$\mathcal{LS}_{14}^0\not\rightarrow \mathcal{LS}_6^{-1}$ & $\dim\mathfrak{D}\left(0,1,-1\right)_1(\g)=8>2=\dim\mathfrak{D}\left(0,1,-1\right)_1(\h)$\\

$\mathcal {LS}_{14}^{-\frac{1}{2}}\not\rightarrow \mathcal{LS}_{13}^{\alpha,\alpha+1}$ & $\mathcal{F}(\g)\simeq\mathcal{LS}_{13}^{-\frac{1}{2},-\frac{1}{2}}\not\rightarrow \h=\mathcal{F}(\h)$\\

$\mathcal {LS}_{14}^{-\frac{1}{2}}\not\rightarrow \mathcal{LS}_{13}^{\alpha,-(\alpha+1)}$ & $\mathcal{F}\left(\g\right)\simeq\mathcal{LS}_{13}^{-\frac{1}{2},-\frac{1}{2}}\not\rightarrow \h=\mathcal{F}(\h)$\\

$\mathcal {LS}_{14}^{-\frac{1}{2}}\not\rightarrow \mathcal{LS}_{16}^{-\frac{1}{2}}$ & $\dim\mathfrak{D}\left(-2,1,-1\right)_1(\g)=2>\dim\mathfrak{D}\left(-2,1,-1\right)_1(\h)$\\

$\mathcal {LS}_{14}^{-\frac{1}{2}}\not\rightarrow \mathcal{LS}_{11},\mathcal{LS}_{12}$ & $\mathcal{F}(\g)\simeq \mathcal{LS}_{13}^{-\frac{1}{2},-\frac{1}{2}}\not\rightarrow \mathcal{LS}_{11}=\mathcal{F}(\h)$\\ \hline


$\mathcal {LS}_{15}^{-\frac{1}{2}}\not\rightarrow \mathcal{LS}_{11}$ & $\mathcal{F}(\g)\simeq \mathcal{LS}_{13}^{-\frac{1}{2},-\frac{1}{2}}\not\rightarrow \mathcal{LS}_{11}=\mathcal{F}(\mathcal{LS}_{11})$\\ 

$\mathcal{LS}_{15}^{\alpha}\not\rightarrow \mathcal{LS}_{2}$ & $\ab(\g)\simeq\mathcal{LS}_{3}\not\rightarrow \h=\ab(\h)$\\

$\mathcal{LS}_{15}^{\alpha}\not\rightarrow \mathcal{LS}_6^{\beta},\mathcal{LS}_{10},\ \mathcal{LS}_{13}^{\beta,\gamma},\ \mathcal{LS}_{16}^{\beta}$ & $\mathcal{F}(\g)\simeq\mathcal{LS}_{13}^{\alpha,-\frac{1}{2}}\not\rightarrow\h=\mathcal{F}(\h)$\\

$\mathcal{LS}_{15}^{\alpha}\not\rightarrow \mathcal {LS}_{15}^{-\frac{1}{2}}$ & $\mathcal{F}(\g)\simeq\mathcal{LS}_{13}^{\alpha,-\frac{1}{2}}\not\rightarrow \mathcal{LS}_{13}^{-\frac{1}{2},-\frac{1}{2}}\simeq\mathcal{F}(\h)$\\ \hline


$\mathcal{LS}_{17}\not\rightarrow \mathcal{LS}_{13}^{\alpha,\alpha+1},\ \mathcal{LS}_{13}^{\alpha,-(\alpha+1)}$ & $\mathcal{F}(\g)\simeq\mathcal{LS}_{16}^{-\frac{1}{2}}\not\rightarrow \h=\mathcal{F}(\h)$\\ \hline


$\mathcal{LS}_{18}^{-\frac{2}{3}}\not\rightarrow \mathcal{LS}_{13}^{-\frac{3}{2},-\frac{1}{2}}$ & $\dim\mathfrak{D}\left(\frac{2}{3},1,-1\right)_1(\g)=1>0=\dim\mathfrak{D}\left(\frac{2}{3},1,-1\right)_1(\h)$\\

$\mathcal{LS}_{18}^{-2}\not\rightarrow \mathcal{LS}_{13}^{\frac{1}{2},-\frac{1}{2}}$ & $\dim\mathfrak{D}\left(-1,1,-1\right)_1(\g)=1>0=\dim\mathfrak{D}\left(-1,1,-1\right)_1(\h)$\\

$\mathcal{LS}_{18}^{-\frac{1}{3}}\not\rightarrow \mathcal{LS}_{13}^{\frac{3}{2},-\frac{1}{2}}$ & $\dim\mathfrak{D}\left(\frac{1}{3},1,-1\right)_1(\g)=1>0=\dim\mathfrak{D}\left(\frac{1}{3},1,-1\right)_1(\h)$\\

$\mathcal{LS}_{18}^{-3}\not\rightarrow \mathcal{LS}_{13}^{\frac{3}{2},-\frac{1}{2}}$ & $\dim\mathfrak{D}\left(3,1,-1\right)_1(\g)=1>0=\dim\mathfrak{D}\left(3,1,-1\right)_1(\h)$\\

$\mathcal{LS}_{18}^1\not\rightarrow \mathcal{LS}_{16}^{\frac{1}{2}},\mathcal{LS}_{13}^{\frac{1}{2},\frac{1}{2}}$ & $\dim\mathfrak{D}\left(\frac{1}{2},1,-1\right)_1(\g)=1>0=\dim\mathfrak{D}\left(\frac{1}{2},1,-1\right)_1(\h)$\\

$\mathcal{LS}_{18}^{\frac{\gamma}{1-\gamma}}\not\rightarrow \mathcal{LS}_{13}^{\gamma,1-\gamma}$, $\gamma\neq\frac{1\pm\sqrt{-3}}{2},0,2$ & $\dim\mathfrak{D}\left(1-\gamma,1,-1\right)_1(\g)=1>0=\dim\mathfrak{D}\left(1-\gamma,1,-1\right)_1(\h)$\\  

$\mathcal{LS}_{18}^{\frac{\gamma}{1-\gamma}}\not\rightarrow \mathcal{LS}_{13}^{\gamma,1-\gamma}$, $\gamma\in\left\{\frac{1\pm\sqrt{-3}}{2},2\right\}$ & $\dim\mathfrak{D}\left(\frac{\gamma}{\gamma-1},1,0\right)_1(\g)=1>0=\dim\mathfrak{D}\left(\frac{\gamma}{\gamma-1},1,0\right)_1(\h)$\\

$\mathcal{LS}_{18}^{\frac{1-\gamma}{\gamma}}\not\rightarrow \mathcal{LS}_{13}^{1-\gamma,\gamma}$,\ $\gamma\neq\frac{1\pm\sqrt{-3}}{2},\pm 1$ & $\dim\mathfrak{D}\left(\gamma,1,-1\right)_1(\g)=1>0=\dim\mathfrak{D}\left(\gamma,1,-1\right)_1(\h)$\\

$\mathcal{LS}_{18}^{\frac{1-\gamma}{\gamma}}\not\rightarrow \mathcal{LS}_{13}^{1-\gamma,\gamma}$,\ $\gamma\in\left\{\frac{1\pm\sqrt{-3}}{2},-1\right\}$ & $\dim\mathfrak{D}\left(\frac{\gamma-1}{\gamma},1,0\right)_1(\g)=1>0=\dim\mathfrak{D}\left(\frac{\gamma-1}{\gamma},1,0\right)_1(\h)$\\

$\mathcal{LS}_{18}^0\not\rightarrow \mathcal{LS}_6^1,\ \mathcal{LS}_{10}$ & $\dim\mathfrak{D}\left(0,1,-1\right)_1(\g)=8>2=\dim\mathfrak{D}\left(0,1,-1\right)_1(\h)$\\

$\mathcal{LS}_{18}^{\alpha}\not\rightarrow \mathcal{LS}_6^{\beta},\ \beta\neq 0,1$ & $\dim\mathfrak{D}\left(1,1,-1\right)_1(\g)=2>0=\dim\mathfrak{D}\left(1,1,-1\right)_1(\h)$\\ \hline


\end{longtable}}

\subsection{Degenerations}

Next we summarize the essential degenerations $\g\rightarrow\h$ in the variety $\mathcal{LS}^{(2,2)}$. In every case, we provide a parametric change of basis $\{x_1^t,x_2^t,y_1^t,y_2^t\}$ of $\g$ such that for $t=0$ we obtain the Lie products of $\h$. For example, consider the first degeneration in Table \ref{Table3}: $\mathcal{LS}_{19}\to\mathcal{LS}_4$. The parameterized base products of $\mathcal{LS}_{19}$ are:
$$\[x_1^t,x_2^t\]=2tx_1^t,\quad\[x_1^t,y_2^t\]=ty_1^t,\quad\[x_2^t,y_1^t\]=-2ty_1^t,\quad\[y_1^t,y_2^t\]=x_1^t,\quad\[y_2^t,y_2^t\]=x_2^t.$$
When $t=0$ we obtain the products of $\mathcal{LS}_4$.

{\scriptsize
\begin{longtable}{|l|llll|}
\caption[]{Degenerations}
\label{Table3}\\
\hline
$\g\to\h$ & \multicolumn{4}{|c|}{Parametrized bases} \\ 
\hline
\endfirsthead
\caption[]{(continued)}\\
\hline
$\g\to\h$ & \multicolumn{4}{|c|}{Parameterized bases} \\ \hline
\endhead
$\mathcal{LS}_{19}\rightarrow \mathcal{LS}_4$ & $x_1^t=te_1$, & $x_2^t=2te_2$, &  $y_1^t=\sqrt{t}f_1$, &  $y_2^t=\sqrt{t}f_2$\\ \hline   

$\mathcal{LS}_{19}\rightarrow \mathcal{LS}_{12}$ & $x_1^t=te_2$, & $x_2^t=te_1$, &  $y_1^t=\sqrt{\frac{t}{2}}f_1$, &  $y_2^t=\sqrt{\frac{t}{2}}f_2$\\ \hline

$\mathcal{LS}_{19}\rightarrow \mathcal{LS}_{18}^{-1}$ &  $x_1^t=e_1$, & $x_2^t=e_2$, &  $y_1^t=tf_1$, &  $y_2^t=tf_2$\\ \hline 

$\mathcal{LS}_{19}\rightarrow \mathcal{LS}_{14}^0$ & $x_1^t=te_1$, & $x_2^t=e_2$, &  $y_1^t=f_2$, &  $y_2^t=tf_1$\\ \hline   

$\mathcal{LS}_1\rightarrow \mathcal{LS}_4$ & $x_1^t=\sqrt{t}\left(-e_1+e_2\right)$, & $x_2^t=e_1+e_2$, &  $y_1^t=\sqrt{t}\left(-f_1+f_2\right)$, &  $y_2^t=f_1+f_2$\\ \hline   

$\mathcal{LS}_5\rightarrow \mathcal{LS}_{9}$ & $x_1^t=e_1+e_2$, & $x_2^t=te_1$, &  $y_1^t=tf_1$, &  $y_2^t=f_1+f_2$\\ \hline  
 
$\mathcal{LS}_5\rightarrow \mathcal{LS}_6^{\alpha}$ & $x_1^t=te_1$, & $x_2^t=e_1+\alpha e_2$, &  $y_1^t=f_1$, &  $y_2^t=f_2$\\ \hline 
  

$\mathcal{LS}_4\rightarrow \mathcal{LS}_2$ & $x_1^t=e_1$, & $x_2^t=t(t-1)e_1+te_2$, &  $y_1^t=f_1$, &  $y_2^t=\frac{1}{2t}f_1+tf_2$\\ \hline   

$\mathcal{LS}_7\rightarrow \mathcal{LS}_2$ & $x_1^t=e_1$, & $x_2^t=te_2$, &  $y_1^t=-\frac{i}{2}f_1+if_2$, &  $y_2^t=\frac{1}{2}f_1+f_2$\\ \hline   

$\mathcal{LS}_7\rightarrow \mathcal{LS}_6^{-1}$ & $x_1^t=^\frac{1}{t}e_1$, & $x_2^t=e_2$, &  $y_1^t=f_1$, &  $y_2^t=f_2$\\ \hline   

$\mathcal{LS}_{7}\rightarrow \mathcal{LS}_{12}$ & $x_1^t=e_1$, & $x_2^t=te_2$, &  $y_1^t=f_1$, &  $y_2^t=\frac{1}{2t}f_1+tf_2$\\ \hline  

$\mathcal{LS}_{8}\rightarrow \mathcal{LS}_2$ & $x_1^t=e_1$, & $x_2^t=te_2$, &  $y_1^t=f_1+f_2$, &  $y_2^t=f_2$\\ \hline  

$\mathcal{LS}_{8}\rightarrow \mathcal{LS}_6^0$ & $x_1^t=e_1$, & $x_2^t=e_2$, &  $y_1^t=f_1$, &  $y_2^t=tf_2$\\ \hline

$\mathcal{LS}_{8}\rightarrow \mathcal{LS}_{12}$ & $x_1^t=e_1$, & $x_2^t=te_2$, &  $y_1^t=f_1$, &  $y_2^t=\frac{1}{t}f_1+f_2$\\ \hline  

$\mathcal{LS}_{9}\rightarrow \mathcal{LS}_{10}$ & $x_1^t=te_1$, & $x_2^t=e_1+e_2$, &  $y_1^t=f_1$, &  $y_2^t=f_2$\\ \hline  

$\mathcal{LS}_{14}^{\alpha}\rightarrow \mathcal{LS}_{13}^{\alpha,-(\alpha+1)}$ & $x_1^t=\frac{1}{t}e_1$, & $x_2^t=e_2$, &  $y_1^t=f_1$, &  $y_2^t=f_2$\\ \hline  

$\mathcal{LS}_{14}^{\alpha}\rightarrow \mathcal{LS}_2$ & $x_1^t=e_1$, & $x_2^t=te_2$, &  $y_1^t=-if_1+\frac{i}{2}f_2$, &  $y_2^t=f_1+\frac{1}{2}f_2$\\ \hline  

$\mathcal{LS}_{15}^{\alpha\neq-\frac{1}{2}}\rightarrow \mathcal{LS}_{12}$ & $x_1^t=e_1$, & $x_2^t=te_2$, &  $y_1^t=f_1$, &  $y_2^t=\frac{2}{t(2\alpha+1)}f_1+f_2$\\ \hline  

$\mathcal{LS}_{15}^{\alpha}\rightarrow \mathcal{LS}_{13}^{\alpha,-\frac{1}{2}}$ & $x_1^t=\frac{1}{t}e_1$, & $x_2^t=e_2$, &  $y_1^t=f_1$, &  $y_2^t=f_2$\\ \hline  


$\mathcal {LS}_{14}^{-\frac{1}{2}}\rightarrow \mathcal {LS}_{15}^{-\frac{1}{2}}$ & $x_1^t=e_1$, & $x_2^t=e_2$, &  $y_1^t=f_1$, &  $y_2^t=tf_2$\\ \hline  

$\mathcal{LS}_{17}\rightarrow \mathcal{LS}_{10}$ & $x_1^t=e_1$, & $x_2^t=te_2$, &  $y_1^t=f_1+f_2$, &  $y_2^t=f_2$\\ \hline  

$\mathcal{LS}_{17}\rightarrow \mathcal{LS}_{12}$ & $x_1^t=t^2e_1$, & $x_2^t=te_2$, &  $y_1^t=t^2f_1$, &  $y_2^t=tf_2$\\ \hline  


$\mathcal{LS}_{17}\rightarrow \mathcal{LS}_{16}^{-\frac{1}{2}}$ &$x_1^t=\frac{1}{t}e_1$, & $x_2^t=e_2$, &  $y_1^t=f_1$, &  $y_2^t=f_2$\\ \hline  

$\mathcal{LS}_{17}\rightarrow \mathcal {LS}_{15}^{-\frac{1}{2}}$ & $x_1^t=e_1$, & $x_2^t=e_2$, &  $y_1^t=f_2$, &  $y_2^t=\frac{1}{t}f_1$\\ \hline  

$\mathcal{LS}_{18}^{\alpha}\rightarrow \mathcal{LS}_{13}^{\alpha,\alpha+1}$ & $x_1^t=e_1$, & $x_2^t=e_2$, &  $y_1^t=\frac{1}{t}f_1$, &  $y_2^t=f_2$\\ \hline  


$\mathcal{LS}_2\rightarrow \mathcal{LS}_3$ & $x_1^t=e_1$, & $x_2^t=e_2$, &  $y_1^t=f_1$, &  $y_2^t=tf_2$\\ \hline  

$\mathcal{LS}_{6}^{\alpha\neq 1}\rightarrow \mathcal{LS}_{11}$ & $x_1^t=\frac{1}{t}e_1$, & $x_2^t=-te_2$, &  $y_1^t=tf_2$, &  $y_2^t=f_1-\frac{1}{\alpha-1}f_2$\\ \hline  

$\mathcal{LS}_{10}\rightarrow \mathcal{LS}_{11}$ & $x_1^t=e_1$, & $x_2^t=te_2$, &  $y_1^t=f_1$, &  $y_2^t=\frac{1}{t}f_2$\\ \hline  

$\mathcal{LS}_{10}\rightarrow \mathcal{LS}_{6}^1$ & $x_1^t=e_1$, & $x_2^t=e_2$, &  $y_1^t=\frac{1}{t}f_1$, &  $y_2^t=f_2$\\ \hline  

$\mathcal{LS}_{12}\rightarrow \mathcal{LS}_{11}$ & $x_1^t=te_1$, & $x_2^t=e_2$, &  $y_1^t=f_1$, &  $y_2^t=f_2$\\ \hline  

$\mathcal{LS}_{12}\rightarrow \mathcal{LS}_{3}$ & $x_1^t=e_1$, & $x_2^t=te_2$, &  $y_1^t=f_2$, &  $y_2^t=f_1$\\ \hline 
 
$\mathcal{LS}_{13}^{\alpha\neq\beta}\rightarrow \mathcal{LS}_{11}$ & $x_1^t=\frac{1}{t}e_1$, & $x_2^t=-te_2$, &  $y_1^t=tf_2$, &  $y_2^t=f_1+\frac{1}{\alpha-\beta}f_2$\\ \hline 
 



$\mathcal{LS}_{16}^{\alpha}\rightarrow \mathcal{LS}_{11}$ & $x_1^t=e_1$, & $x_2^t=te_2$, &  $y_1^t=f_1$, &  $y_2^t=\frac{1}{t}f_2$\\ \hline  

$\mathcal{LS}_{16}^{\alpha}\rightarrow \mathcal{LS}_{13}\left(\alpha,\alpha\right)$ & $x_1^t=e_1$, & $x_2^t=e_2$, &  $y_1^t=\frac{1}{t}f_1$, &  $y_2^t=f_2$\\ \hline  

 
$\mathcal {LS}_{15}^{-\frac{1}{2}}\rightarrow \mathcal{LS}_3$ & $x_1^t=e_1$, & $x_2^t=te_2$, &  $y_1^t=f_1$, &  $y_2^t=f_2$\\ \hline  

$\mathcal {LS}_{15}^{-\frac{1}{2}}\rightarrow \mathcal{LS}_{13}^{-\frac{1}{2},-\frac{1}{2}}$ & $x_1^t=\frac{1}{t}e_1$, & $x_2^t=e_2$, &  $y_1^t=f_1$, &  $y_2^t=f_2$\\ \hline  



\end{longtable}}

\normalsize

With all this, we can draw the Hasse diagram of essential degenerations for the variety $\mathcal{LS}^{(2,2)}$:

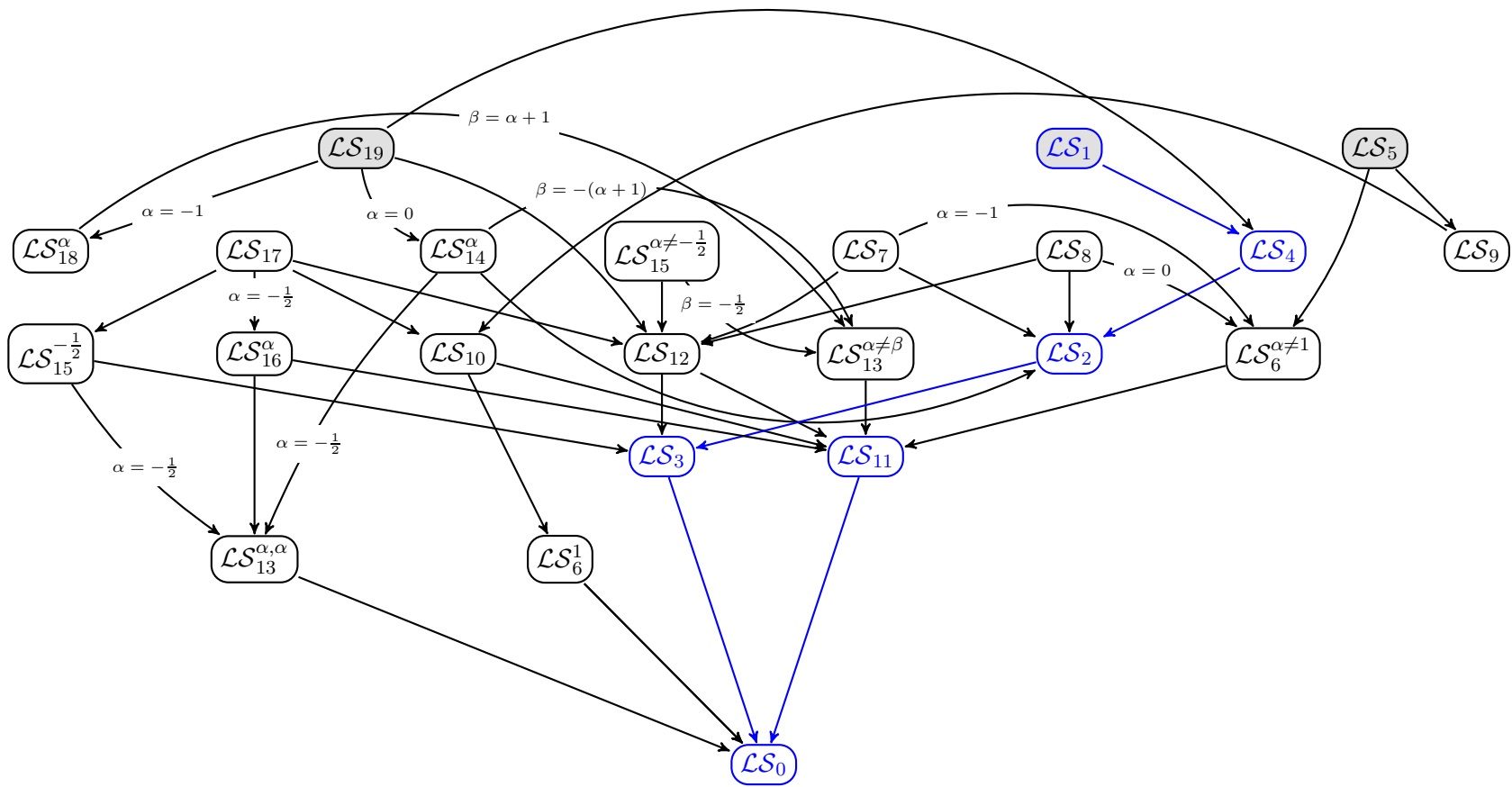
\begin{landscape}
{\begin{center}

\begin{tikzpicture}[->,>=stealth',shorten >=0.01cm,auto,node distance=1.5cm,
                    thick,main node/.style={rectangle,draw,fill=white!12,rounded corners=1.5ex,font=\sffamily \bf \bfseries },
                    blue node/.style={rectangle,draw, color=blue,fill=white!12,rounded corners=1.5ex,font=\sffamily \bf \bfseries },
                    thick,rigid node/.style={rectangle,draw,fill=black!12,rounded corners=1.5ex,font=\sffamily \bf \bfseries },
                    thick,bluerigid node/.style={rectangle,draw,color=blue,fill=black!12,rounded corners=1.5ex,font=\sffamily \bf \bfseries },
                    style={draw,font=\sffamily \scriptsize \bfseries }]
                    
\node (71)   {};

\node (61) [below          of=71]      {};
\node (51) [below         of=61]      {};
\node (41) [below         of=51]      {};
\node (31) [below         of=41]      {};
\node (21) [below         of=31]      {};
\node (11) [below         of=21]      {};
\node (01) [below         of=11]      {};

\node (62)  [right of =61]                      { };
\node (63)  [right of =62]                      { };
\node  (64)  [right of =63]                      { };
\node (65)  [right of =64]                      {};
\node (66)  [right of =65]                      { };
\node  (67)  [right of =66]                      { };
\node  (68)  [right of =67]                      { };
\node (69)  [right of =68]                      { };
\node (610)  [right of =69]                      { };
\node (611)  [right of =610]                      { };
\node  (612)  [right of =611]                      { };

\node [rigid node] (19)  [right of =62]                      {$\mathcal{LS}_{19}$};
\node [bluerigid node] (1)  [right of =69]                      {$\mathcal{LS}_1$ };
\node [rigid node] (5)  [right of =612]                      {$\mathcal{LS}_5$};

\node (52)  [right of =51]                      { };
\node (53)  [right of =52]                      { };
\node  (54)  [right of =53]                      { };
\node (55)  [right of =54]                      {};
\node (56)  [right of =55]                      { };
\node  (57)  [right of =56]                      { };
\node  (58)  [right of =57]                      { };
\node (59)  [right of =58]                      { };
\node (510)  [right of =59]                      { };
\node (511)  [right of =510]                      { };
\node  (512)  [right of =511]                      { };
\node (513)  [right of =512]                      { };
\node  (514)  [right of =513]                      { };

\node [main node] (18a)  [left of =51]                      {$\mathcal{LS}_{18}^{\alpha}$};
\node [main node] (17)  [right of =51]                      {$\mathcal{LS}_{17}$};
\node [main node] (14a)  [right of =53]                      {$\mathcal{LS}_{14}^{\alpha}$};
\node [main node] (15a)  [right of =55]                      {$\mathcal{LS}_{15}^{\alpha\neq-\frac{1}{2}}$};
\node [main node] (7)  [right of =57]                      {$\mathcal{LS}_7$};
\node [main node] (8)  [right of =59]                      {$\mathcal{LS}_8$};
\node [blue node] (4)  [right of =511]                      {$\mathcal{LS}_4$};
\node [main node] (9)  [right of =513]                      {$\mathcal{LS}_9$};


\node (42)  [right of =41]                      { };
\node (43)  [right of =42]                      { };
\node  (44)  [right of =43]                      { };
\node (45)  [right of =44]                      {};
\node (46)  [right of =45]                      { };
\node  (47)  [right of =46]                      { };
\node  (48)  [right of =47]                      { };
\node (49)  [right of =48]                      { };
\node (410)  [right of =49]                      { };
\node (411)  [right of =410]                      { };
\node  (412)  [right of =411]                      { };
\node  (413)  [right of =412]                      { };

\node [main node] (1512)  [left of =41]                       {$\mathcal{LS}_{15}^{-\frac{1}{2}}$};
\node [main node] (16a)  [right of =41]                      {$\mathcal{LS}_{16}^{\alpha}$};
\node [main node] (10)  [right of =43]                      {$\mathcal{LS}_{10}$};
\node [main node] (12)  [right of =45]                      {$\mathcal{LS}_{12}$};
\node [main node] (13ab)  [right of =47]                     {$\mathcal{LS}_{13}^{\alpha\neq\beta}$};
\node [blue node] (2)  [right of =49]                     {$\mathcal{LS}_2$};
\node [main node] (6a)  [right of =411]                     {$\mathcal{LS}_6^{\alpha\neq 1}$};

\node (32)  [right of =31]                      { };
\node (33)  [right of =32]                      { };
\node  (34)  [right of =33]                      { };
\node (35)  [right of =34]                      { };
\node (36)  [right of =35]                      { };
\node  (37)  [right of =36]                      { };
\node  (38)  [right of =37]                      { };
\node (39)  [right of =38]                      { };
\node (310)  [right of =39]                      { };
\node (311)  [right of =310]                      { };
\node  (312)  [right of =311]                      { };

\node [blue node] (3)  [right of =35]                      {$\mathcal{LS}_3$};
\node [blue node] (11)  [right of =37]                       {$\mathcal{LS}_{11}$};

\node (22)  [right of =21]                      { };
\node (23)  [right of =22]                      { };
\node  (24)  [right of =23]                      { };
\node (25)  [right of =24]                      {};
\node (26)  [right of =25]                      { };
\node  (27)  [right of =26]                      { };
\node  (28)  [right of =27]                      { };
\node (29)  [right of =28]                      { };
\node (210)  [right of =29]                      { };
\node (211)  [right of =210]                      { };
\node  (212)  [right of =211]                      { };

\node [main node] (61)  [right of =24]                      {$\mathcal{LS}_6^1$};
\node [main node] (13aa)  [right of =21]                       {$\mathcal{LS}_{13}^{\alpha,\alpha}$};

\node (02)  [right of =01]                      { };
\node (03)  [right of =02]                      { };
\node  (04)  [right of =03]                      { };
\node (05)  [right of =04]                      {};
\node (06)  [right of =05]                      { };
\node  [blue node] (0)  [right of =06]                      {$\mathcal{LS}_0$ };

\path[every node/.style={font=\sffamily\tiny}]

(19)  edge [bend right=0, color=black] node[above=-5, right=-30,  fill=white]{$\alpha=-1$}  (18a)
(19)  edge [bend right=30, color=black] node[above=0, right=-10,  fill=white]{$\alpha=0$}  (14a)
(19)  edge [bend right=-40, color=black] node{}  (4)
(19)  edge [bend right=-20, color=black] node{}  (12)

(1)  edge [bend right=0, color=blue] node{}  (4)

(5)  edge [bend right=0, color=black] node{}  (9)
(5)  edge [bend right=-10, color=black] node{}  (6a)

(18a)  edge [bend right=-45, color=black] node[above=0, right=-10,  fill=white]{$\beta=\alpha+1$}  (13ab)

(17)  edge [bend right=0, color=black] node{}  (10)
(17)  edge [bend right=0, color=black] node{}  (12)
(17)  edge [bend right=0, color=black] node[above=2, right=-15,  fill=white]{$\alpha=-\frac{1}{2}$}  (16a)
(17)  edge [bend right=0, color=black] node{}  (1512)

(14a)  edge [bend right=35, color=black] node{}  (2)
(14a)  edge [bend right=-50, color=black] node[above=2, right=-70,  fill=white]{$\beta=-(\alpha+1)$}  (13ab)
(14a)  edge [bend right=10, color=black] node[above=-20, right=-30,  fill=white]{$\alpha=-\frac{1}{2}$}  (13aa)

(4)  edge [bend right=0, color=blue] node{}  (2)

(7)  edge [bend right=0, color=black] node{}  (2)
(7)  edge [bend right=-40, color=black] node[above=0, right=-70,  fill=white]{$\alpha=-1$}  (6a)
(7)  edge [bend right=-10, color=black] node{}  (12)

(8)  edge [bend right=0, color=black] node{}  (2)
(8)  edge [bend right=-10, color=black] node[above=7, right=-25,  fill=white]{$\alpha=0$}  (6a)
(8)  edge [bend right=0, color=black] node{}  (12)

(9)  edge [bend right=40, color=black] node{}  (10)

(15a)  edge [bend right=0, color=black] node{}  (12)
(15a)  edge [bend right=25, color=black] node[above=12, right=-30,  fill=white]{$\beta=-\frac{1}{2}$}  (13ab)

(10)  edge [bend right=0, color=black] node{}  (11)
(10)  edge [bend right=0, color=black] node{}  (61)

(12)  edge [bend right=0, color=black] node{}  (11)
(12)  edge [bend right=0, color=black] node{}  (3)

(13ab)  edge [bend right=0, color=black] node{}  (11)

(2)  edge [bend right=0, color=blue] node{}  (3)

(6a)  edge [bend right=0, color=black] node{}  (11)

(61)  edge [bend right=0, color=black] node{}  (0)

(1512)  edge [bend right=0, color=black] node{}  (3)
(1512)  edge [bend right=10, color=black] node[above=0, right=-15,  fill=white]{$\alpha=-\frac{1}{2}$}  (13aa)

(16a)  edge [bend right=0, color=black] node{}  (11)
(16a)  edge [bend right=0, color=black] node{}  (13aa)

(3)  edge [bend right=0, color=blue] node{}  (0)

(11)  edge [bend right=0, color=blue] node{}  (0)

(61)  edge [bend right=0, color=black] node{}  (0)

(13aa)  edge [bend right=0, color=black] node{}  (0);
\end{tikzpicture}
\end{center}
}
\end{landscape}


\subsection{The irreducible components of $\mathcal{LS}^{(2,2)}$}

In order to describe the irreducible components which are closures of single orbits, we must obtain the rigid Lie superalgebras in $\mathcal{LS}^{(2,2)}$. Analogous to the Lie algebras case, it follows that if the cohomology group $(H^2(\g, \g))_0=0$ then the Lie superalgebra $\g$ is rigid.  Then we obtain:

\begin{theorem}\label{thm:rigid}
There are 3 rigid Lie superalgebras in the variety $\mathcal{LS}^{(2,2)}$: $\mathcal{LS}_{19}$, $\mathcal{LS}_1$ and $\mathcal{LS}_5$. Moreover, $\mathcal{LS}_1$ is a rigid nilpotent Lie superalgebra.
\end{theorem}

\begin{proof}
The second cohomology groups are:
\begin{itemize}
\item $H^2(\mathcal{LS}_{19},\mathcal{LS}_{19})=\langle 0\rangle\oplus \langle e^1\wedge e^2\otimes f_1-e^1\wedge f^2\otimes e_1+e^2\wedge f^1\otimes e_1\rangle$.
\item $H^2(\mathcal{LS}_1,\mathcal{LS}_1)=\langle 0\rangle\oplus \langle 2e^1\wedge f^2\otimes e_1+f^1\wedge f^2\otimes f_1,2e^2\wedge f^1\otimes e_2+f^1\wedge f^2\otimes f_2\rangle$.
\item $H^2(\mathcal{LS}_5,\mathcal{LS}_5)=0$.
\end{itemize}
\end{proof}
The degenerations in the next table, will help us to obtain the irreducible components which are closure of the union of infinite families of Lie superalgebras. Suppose $\mathcal{L}=\{\g(\alpha)\}_{\alpha}$ is an infinite family of Lie superalgebras, then we consider $\alpha$ ($=\alpha(t)$) as parameterized by $t$  and construct a degeneration from $\g(\alpha(t))\to\h$ as usual.

{\scriptsize
\begin{longtable}{|l|llll|}
\caption[]{Degenerations}
\label{Table4}\\
\hline
$\g(\alpha)\to\h$ & \multicolumn{4}{|c|}{Parametrized bases} \\ 
\hline
\endfirsthead
\caption[]{(continued)}\\
\hline
$\g\to\h$ & \multicolumn{4}{|c|}{Parameterized bases} \\ \hline
\endhead

$\mathcal{LS}_{14}^{\alpha(t)}\rightarrow \mathcal{LS}_{17}$ & $x_1^t=e_1$, & $x_2^t=e_2$, &  $y_1^t=-i\sqrt{t}f_1$, &  $y_2^t=\frac{1}{2}f_1+f_2$\\ 
$\alpha(t)=-\frac{1}{2}-i\sqrt{t}$ &&&&  \\ \hline

$\mathcal{LS}_{14}^{\alpha(t)}\rightarrow \mathcal{LS}_{7}$ & $x_1^t=e_1$, & $x_2^t=-te_2$, &  $y_1^t=f_1$, &  $y_2^t=f_2$\\ 
$\alpha(t)=-\frac{1}{t}$ &&&&  \\ \hline

$\mathcal{LS}_{15}^{\alpha(t)}\rightarrow \mathcal{LS}_{8}$ & $x_1^t=e_1$, & $x_2^t=-2te_2$, &  $y_1^t=f_1$, &  $y_2^t=f_2$\\ 
$\alpha(t)=-\frac{1}{2t}$ &&&&  \\ \hline

$\mathcal{LS}_{18}^{\alpha(t)}\rightarrow \mathcal{LS}_{9}$ & $x_1^t=-te_2$, & $x_2^t=e_1$, &  $y_1^t=f_1$, &  $y_2^t=f_2$\\ 
$\alpha(t)=-\left(\frac{1+t}{t}\right)$ &&&&  \\ \hline

$\mathcal{LS}_{13}^{\alpha(t),\beta(t)}\rightarrow \mathcal{LS}_{6}^{\gamma}$ & $x_1^t=e_1$, & $x_2^t=te_2$, &  $y_1^t=f_1$, &  $y_2^t=f_2$\\ 
$\alpha(t)=\frac{1}{t},\ \beta(t)=\frac{\gamma}{t}$ &&&&  \\ \hline

$\mathcal{LS}_{13}^{\alpha(t),\beta(t)}\rightarrow \mathcal{LS}_{16}^{\gamma}$ & $x_1^t=e_1$, & $x_2^t=e_2$, &  $y_1^t=2\sqrt{t}f_1$, &  $y_2^t=f_1+f_2$\\ 
$\alpha(t)=\frac{\gamma+\sqrt{t}}{1+t},\ \beta(t)=\frac{\gamma-\sqrt{t}}{1+t}$ &&&&  \\ \hline

\end{longtable}}

\normalsize

With all this, we can enunciate the main theorem of this Section:

\begin{theorem}
The irreducible components of the variety $\mathcal{LS}^{(2,2)}$ are:
\begin{itemize}
\item $\mathcal{C}_1=\overline{\mathcal{O}(\mathcal{LS}_{19})}$,
\item $\mathcal{C}_2=\overline{\mathcal{O}(\mathcal{LS}_{1})}$,
\item $\mathcal{C}_3=\overline{\mathcal{O}(\mathcal{LS}_{5})}$,
\item $\mathcal{C}_4=\displaystyle\overline{\bigcup_{\alpha}\mathcal{O}(\mathcal{LS}_{14}^{\alpha})}$,
\item $\mathcal{C}_5=\displaystyle\overline{\bigcup_{\alpha}\mathcal{O}(\mathcal{LS}_{15}^{\alpha})}$,
\item $\mathcal{C}_6=\displaystyle\overline{\bigcup_{\alpha}\mathcal{O}(\mathcal{LS}_{18}^{\alpha})}$,
\item $\mathcal{C}_7=\displaystyle\overline{\bigcup_{\alpha,\beta}\mathcal{O}(\mathcal{LS}_{13}^{\alpha,\beta})}$.
\end{itemize}
\label{2-2-irred-comp}
\end{theorem}

\section*{Aknowledgements} The first author was supported by grant FOMIX- CONACYT  YUC-2013-C14-221183 and Becas Iberoam\'erica de J\'ovenes Profesores e Investigadores, Santander Universidades. The second author was supported by grants FOMIX-CONACYT  YUC-2013-C14-221183 and 222870. The first author express her gratitude to Universidad Aut\'onoma de Yucat\'an and Centro de Investigaci\'on en Matem\'aticas - Unidad M\'erida for their hospitality during her research stays in both centers. Both authors thank Ivan Kaygorodov for useful comments about the presentation of this paper.

\end{document}